\documentclass[final]{siamltex}
\usepackage{palatino}
\usepackage{epsfig}
\usepackage{ifthen}
\usepackage{latexsym}
\usepackage{amssymb}

\renewcommand{\baselinestretch}{2.0}

\newtheorem{conjecture}[theorem]{Conjecture}
\newtheorem{question}{Question}

\begin{document}
\title{Generalizations of Graham's Pebbling Conjecture}
\author{David S. Herscovici\footnote{
Department of Computer Science/IDD,
CL-AC1,
Quinnipiac University,
275 Mount Carmel Avenue,
Hamden, CT 06518,
\texttt{David.Herscovici@quinnipiac.edu}}
\and Benjamin D. Hester\footnote{
Department of Mathematics and Statistics,
Arizona State University,
Tempe, AZ, 85287,
\texttt{benjamin@mathpost.la.asu.edu}}
\and Glenn H. Hurlbert\footnote{
Department of Mathematics and Statistics,
Arizona State University,
Tempe, AZ, 85287,
\texttt{hurlbert@asu.edu}}\\
}
\renewcommand{\baselinestretch}{1.0}
\maketitle
\renewcommand{\baselinestretch}{2.0}

\begin{abstract}
We investigate generalizations of pebbling numbers and of Graham's
pebbling conjecture that $\pi(G \times H) \leq \pi(G) \pi(H)$, where $\pi(G)$
is the pebbling number of the graph $G$.  We develop new machinery to
attack the conjecture, which is now twenty years old.  We show that
certain conjectures imply others that initially appear stronger.  We
also find counterexamples that show that Sj\"ostrand's theorem on
cover pebbling does not apply if we allow the cost of transferring a
pebble from one vertex to an adjacent vertex to depend on the edge and
we describe an alternate pebbling number for which Graham's conjecture
is demonstrably false.
\end{abstract}

\begin{keywords} 
Pebbling, Graham's conjecture
\end{keywords}

\begin{AMS}
05C99
\end{AMS}

\pagestyle{myheadings}
\thispagestyle{plain}
\markboth{D. HERSCOVICI, B. HESTER AND G. HURLBERT}{GENERALIZATIONS OF
  GRAHAM'S PEBBLING CONJECTURE}

\section{Distributions and Pebbling Numbers}
\label{distros}

We investigate various generalizations of Graham's pebbling conjecture and
relationships between those generalizations. \\
\textbf{Definition}: Chung defined a \emph{distribution} of pebbles on
a graph $G = (V, E)$ as a placement of pebbles on the vertices of the
graph.  Equivalently, a distribution $D$ is a function $D : V(G)
\rightarrow \mathbb{N}$, where $D(v)$ represents the number of pebbles
on the vertex $v$.  Also, for every distribution $D$ and every
positive integer $t$, we define $tD$ as the distribution given by
$(tD) (v) = t D(v)$ for every vertex $v$ in $G$.
Following~\cite{Hurlbert}, we also define $|D|$ as the total number of
pebbles in the distribution $D$. \\
\textbf{Definition}: A \emph{pebbling move} consists of removing two
pebbles from some vertex, throwing one of the pebbles away, and moving
the other pebble to an adjacent vertex.

The following definitions are motivated by Section~4 in~\cite{all
cycles}: \\
\textbf{Definition}: Given two distributions $D'$ and $D''$ on a graph
$G$, we say $D''$ \emph{contains} $D'$ if $D'(v) \leq D''(v)$ for
every vertex $v \in V(G)$. \\
\textbf{Definition}: Given two distributions $D$ and $D'$ on a graph $G$,
we say that $D'$ is \emph{reachable} from $D$ if it is possible to use
a sequence of pebbling moves to go from $D$ to a distribution $D''$
that contains $D'$.

We refer to distributions that we are trying to reach as \emph{target
distributions}.  Some authors have called such distributions
\emph{weight functions}~\cite{Crull et. al., Vuong and Wychkoff}, and
speak of \emph{weighted cover pebbling numbers}.  We avoid this
terminology; instead, following~\cite{Elledge and Hurlbert}, we use
the term \emph{weighted graphs} to refer to graphs whose edges are
weighted (see Section~\ref{weighted pebbling}).

We define our most general pebbling number on unweighted graphs as
follows. \\
\textbf{Definition}: Let $\mathcal{S}$ be a set of distributions on a
graph $G$.  Then the \emph{pebbling number of $\mathcal{S}$ in $G$},
denoted $\pi(G, \mathcal{S})$ is the smallest number such that every
distribution $D \in \mathcal{S}$ is reachable from every distribution
that starts with $\pi(G, \mathcal{S})$ (or more) pebbles on $G$.

It is customary to require the graph $G$ to be connected and
undirected, but we may dispense with this requirement and allow
$\pi(G,~\mathcal{S})~=~\infty$ if some distribution in $\mathcal{S}$ is
unreachable from distributions with arbitrarily many pebbles.  In
particular, Moews~\cite{Moews} considered trees to be directed graphs
with all edges directed toward the target vertex.

There are several ways to specialize the above definition. \\
\textbf{Definition}: Let $D$ be a distribution of pebbles on a graph
$G$.  Then the \emph{pebbling number of $D$ in $G$}, denoted $\pi(G,
D)$, is defined by $\pi(G, D) = f (G, \{ D \})$, i.~e.\ the smallest
number such that $D$ is reachable from every distribution that starts
with $\pi(G, D)$ pebbles on $G$.

We define some specific distributions and sets of distributions. \\
\textbf{Definition}: For any vertex $v \in V(G)$, we define the
distribution $\delta_v$ as the function
\[
\delta_v(x) = \left\{
\begin{array}{cl} 1, & x = v \\
                  0, & x \neq v \end{array}
\right.
\]
We also define $\mathcal{S}_t(G) = \{ t \delta_v : v \in V(G)
\}$---the set of distributions with $t$ pebbles on a single vertex.

The definitions of pebbling numbers in the remainder of this section
are consistent with the definitions given by Chung~\cite{Hypercubes}
and the rest of the literature on pebbling, but we give definitions in
terms of the previous definitions. \\
\textbf{Definition}: Choose $v \in V(G)$.  Then the \emph{pebbling number
of $v$ in $G$}, denoted $\pi(G, v)$, is defined by $\pi(G, v) = f (G,
\delta_v)$.  Thus, $\pi(G, v)$ is the smallest number such that the
vertex $v$ can be reached from every distribution of $\pi(G, v)$ pebbles
on $G$. \\
\textbf{Definition}: The \emph{pebbling number of $G$} is defined as
$\pi(G) = f (G, \mathcal{S}_1(G))$.  Thus, $\pi(G)$ is the smallest number
such that any single vertex is reachable from every distribution of
$\pi(G)$ pebbles on $G$. \\
\textbf{Definition}: For any $v \in V(G)$ and any positive integer $t$,
the \emph{$t$-pebbling number of $v$ in $G$}, denoted $\pi_t(G, v)$, is
defined by $\pi_t(G, v) = \pi(G, t \delta_v)$.  Thus, $\pi_t(G, v)$ is the
smallest number such that $t$ pebbles can be moved to the vertex $v$
from every distribution of $\pi_t(G, v)$ pebbles on $G$. \\
\textbf{Definition}: The \emph{$t$-pebbling number of $G$} is defined
as $\pi(G) = f (G, \mathcal{S}_t(G))$.  Thus, $\pi(G)$ is the smallest
number such that $t$ pebbles can be moved to any single vertex from
every distribution of $\pi_t(G)$ pebbles on $G$.

Proposition~\ref{relationships} notes some straightforward
relationships between these definitions.

\begin{proposition}
Let $G$ be any graph, and let $\mathcal{S}$ and $\mathcal{S}'$ be two
sets of distributions on $G$.  Then the various pebbling numbers are
related as follows.

\begin{enumerate}
\item We have ${\displaystyle \pi(G, \mathcal{S}) = \max_{D \in {\cal
S}} \pi(G, D)}$.
\item In particular, we have ${\displaystyle \pi(G) = \max_{v \in V(G)}
\pi(G, v)}$, and ${\displaystyle \pi_t(G) = \max_{v \in V(G)} \pi_t(G, v)}$.
\item Furthermore, if $\mathcal{S} \subseteq \mathcal{S}'$, then $\pi(G,
{\cal S}) \leq \pi(G, \mathcal{S}')$.
\label{subsets}
\end{enumerate}
\label{relationships}
\end{proposition}

The \emph{cover pebbling number} was first defined by Crull
et.~al.~\cite{Crull et. al.}.  We define it as follows. 

We also define the distribution $\Gamma_G$ as the constant function
$\Gamma_G(x) = 1$ for every vertex $x$ in $V(G)$.\\
\textbf{Definition}: The \emph{cover pebbling number of $G$} is defined
as $\gamma(G) = \pi(G, \Gamma_G)$.  Thus, $\gamma(G)$ is the smallest
number such that one pebble can be moved to every vertex
simultaneously from every distribution of $\gamma(G)$ pebbles on $G$.

Sj\"ostrand~\cite{Sjostrand} proved Theorem~\ref{cover pebbling
number}.

\begin{theorem}[Sj\"ostrand]
If $D$ is a distribution of pebbles on the graph $G$ such that $D(v)
\geq 1$ for every vertex $v$ in $V(G)$, then $\pi(G, D)$ is the smallest
number $n$ with the property that if $n$ pebbles are placed on a
\emph{single} vertex, then $D$ is reachable, regardless of which
vertex contained the initial pebbles.  Thus, we only have to look at
starting distributions in which all pebbles are on the same vertex.
In particular, this allows us to compute $\gamma(G)$ easily.
\label{cover pebbling number}
\end{theorem}

\section{Cartesian products}
\label{products}

\textbf{Definition}: If $G = (V_{G}, E_{G})$ and $H = (V_{H}, E_{H})$
are two graphs, their Cartesian product is the graph $G\times H$ whose
vertex set is the product
\[
V_{G\times H} = V_{G} \times V_{H}=\{(x, y) : x \in V_{G}, y \in
V_{H}\},
\]
and whose edges are given by
\[
E_{G\times H} = \{((x, y), (x, y')) : (y, y') \in E_{H} \} \cup \{((x,
y), (x', y)) : (x, x') \in E_{G} \}.
\]

We first define the product of two distributions.  This definition
appeared with slightly different notation in~\cite{all cycles}. \\
\textbf{Definition}: If $D_g$ and $D_h$ are distributions on $G$ and
$H$ respectively, then we define $D_g \cdot D_h$ as the distribution
on $G \times H$ such that
\[ (D_g \cdot D_h) ((x, y)) = D_g (x) D_h (y) \]
for every vertex $(x, y) \in V(G \times H)$.  Similarly, if
$\mathcal{S}_G$ and $\mathcal{S}_H$ are sets of distributions on $G$
and $H$ respectively, then $\mathcal{S}_G \cdot \mathcal{S}_H$ is the
set of distributions on $G \times H$ given by
\[
\mathcal{S}_G \cdot \mathcal{S}_H = \{ D_g \cdot D_h : D_g \in
\mathcal{S}_G \mbox{ and } D_h \in \mathcal{S}_H \}
\]
The following conjectures generalize Graham's Conjecture
(Conjecture~\ref{Graham's conjecture}).

\begin{conjecture}
For all graphs $G$ and $H$, and all sets of distributions
$\mathcal{S}_G$ and $\mathcal{S}_H$ on $G$ and $H$ respectively, we
have $\pi(G \times H, \mathcal{S}_G \cdot \mathcal{S}_H) \leq \pi(G,
\mathcal{S}_G) \pi(H, \mathcal{S}_H)$.
\label{sets of distributions}
\end{conjecture}

By choosing specific sets of distributions $\mathcal{S}_G$ and
$\mathcal{S}_H$, Conjecture~\ref{sets of distributions} generates
several more conjectures.  Conjecture~\ref{single distributions} first
appeared as Conjecture~4.1 in~\cite{all cycles}.

\begin{conjecture}
For all graphs $G$ and $H$, and all distributions $D_g$ and $D_h$ on
$G$ and $H$ respectively, we have $\pi(G \times H, D_g \cdot D_h) \leq
\pi(G, D_g) \pi(H, D_h)$.
\label{single distributions}
\end{conjecture}

In particular, Sj\"ostrand~\cite{Sjostrand} proved Theorem~\ref{cover
product} as a consequence of Theorem~\ref{cover pebbling number}.

\begin{theorem}
Let $D_g$ be a distribution on the graph $G$ such that $D_g(v) \geq 1$
for every vertex $v$ in $V(G)$, and let $D_h$ be a distribution on $H$
with the same property.  Then $\pi(G \times H, D_g \cdot D_h) \leq \pi(G,
D_g) \pi(H, D_h)$.
\label{cover product}
\end{theorem}

For positive integers $s$ and $t$, and vertices $x \in V(G)$ and $y
\in V(H)$, we let $D_g = s \delta_x$ and $D_h = t \delta_y$ in
Conjecture~\ref{single distributions} to obtain
Conjecture~\ref{st-pebbling vertices}.

\begin{conjecture} 
For all graphs $G$ and $H$, all positive integers $s$ and $t$, and all
vertices $x \in V(G)$ and $y \in V(H)$, we have $\pi_{st}(G \times H,
(x, y)) \leq \pi_s(G, x) \pi_t(H, y)$.
\label{st-pebbling vertices}
\end{conjecture}

Letting $s = t = 1$ in Conjecture~\ref{st-pebbling vertices}, we can
specialize to Conjecture~\ref{pebbling vertices}, which first appeared
in~\cite{C5xC5}.

\begin{conjecture}
For all graphs $G$ and $H$ and all vertices $x \in V(G)$ and $y \in
V(H)$, we have $\pi(G \times H, (x, y)) \leq \pi(G, x) \pi(H, y)$.
\label{pebbling vertices}
\end{conjecture}

By not specifying a target vertex, we postulate
Conjectures~\ref{st-pebbling graphs} and~\ref{Graham's conjecture}.
Conjecture~\ref{st-pebbling graphs} first appeared in~\cite{many
cycles}, and Chung~\cite{Hypercubes} attributed
Conjecture~\ref{Graham's conjecture} to Graham.

\begin{conjecture}
For all graphs $G$ and $H$, all positive integers $s$ and $t$, we have
$\pi_{st}(G \times H) \leq \pi_s(G) \pi_t(H)$.
\label{st-pebbling graphs}
\end{conjecture}

\begin{conjecture}[Graham's Conjecture]
For all graphs $G$ and $H$, we have $\pi(G \times H) \leq \pi(G) \pi(H)$.
\label{Graham's conjecture}
\end{conjecture}

\section{Equivalent conjectures}

We now establish some equivalences and logical relationships among the
Conjectures from Section~\ref{products}.  We first note that
Conjectures~\ref{sets of distributions} and~\ref{single distributions}
are equivalent.  We then use a similar argument to show that
Conjectures~\ref{st-pebbling vertices} and~\ref{pebbling vertices}
imply Conjectures~\ref{st-pebbling graphs} and~\ref{Graham's
conjecture}, respectively.  We then establish equivalences within
Conjecture~\ref{st-pebbling vertices} for different values of $s$ and
$t$.  In particular, we show that we can factor out powers of two.
This suggests two more conjectures, one that is equivalent to
Conjecture~\ref{st-pebbling vertices}, and another that is equivalent
to Conjecture~\ref{pebbling vertices}.

\begin{proposition}
Let $G$ and $H$ be fixed graphs.  Then the following conjectures are
equivalent:
\begin{enumerate}
\item $\pi(G \times H, \mathcal{S}_G \cdot \mathcal{S}_H) \leq \pi(G,
\mathcal{S}_G) \pi(H, \mathcal{S}_H)$ for all sets of distributions
$\mathcal{S}_G$ on $G$ and $\mathcal{S}_H$ on $H$.
\label{sets}
\item $\pi(G \times H, D_g \cdot D_h) \leq \pi(G, D_g) \pi(H, D_h)$ for all
individual distributions $D_g$ on $G$ and $D_h$ on $H$.
\label{singles}
\end{enumerate}

In particular, Conjectures~\ref{sets of distributions} and~\ref{single
distributions} are equivalent.
\label{single distributions implies all sets}
\end{proposition}

\begin{proof}
If statement~\ref{sets} holds, applying it with
$\mathcal{S}_G = \{ D_g \}$ and $\mathcal{S}_H = \{ D_h \}$ implies
statement~\ref{singles}.  Conversely, if statement~\ref{singles}
holds, we note that from Proposition~\ref{relationships}, we have
\[
\pi(G \times H, \mathcal{S}_G \cdot \mathcal{S}_H) = \max_{D \in
\mathcal{S}_G \cdot \mathcal{S}_H} \pi(G \times H, D).
\]
Let $D = D_g \cdot D_h$ be a distribution for which this maximum is
achieved and apply statement~\ref{singles} to obtain
\[
\pi(G \times H, \mathcal{S}_G \cdot \mathcal{S}_H) = \pi(G \times H, D_g
\cdot D_h) \leq \pi(G, D_g) \pi(H, D_h).
\]
Clearly, this product is at most
\[
\max_{D_g \in \mathcal{S}_G}
\pi(G, D_g) \max_{D_h \in \mathcal{S}_H} \pi(H, D_h)
=\pi(G, \mathcal{S}_G) \pi(H, \mathcal{S}_H),
\]
by Proposition~\ref{relationships}.
\end{proof}

Proposition~\ref{all vertices implies graph} shows that
Conjecture~\ref{st-pebbling graphs} and Conjecture~\ref{Graham's
conjecture} follow from Conjectures~\ref{st-pebbling vertices} and
Conjecture~\ref{pebbling vertices}, respectively.

\begin{proposition}
Let $G$ and $H$ be graphs and let $s$ and $t$ be positive integers
with the property that $\pi_{st}(G \times H, (x, y)) \leq \pi_s(G, x)
\pi_t(H, y)$ for every pair of vertices $x \in V(G)$ and $y \in V(H)$.
Then $\pi_{st}(G \times H) \leq \pi_s(G) \pi_t(H)$.  Thus,
Conjecture~\ref{st-pebbling vertices} implies
Conjecture~\ref{st-pebbling graphs} and Conjecture~\ref{pebbling
vertices} implies Conjecture~\ref{Graham's conjecture}.
\label{all vertices implies graph}
\end{proposition}

\begin{proof}
From Proposition~\ref{relationships}, we know
\[
\pi_{st} (G \times H) = \max_{(x, y) \in V(G \times H)} \pi_{st}(G \times
H, (x, y)).
\]
Let $(x, y)$ be a vertex for which this maximum is achieved.  Then
\[
\pi_{st}(G \times H) = \pi_{st}(G \times H, (x, y)) \leq 
\pi_s(G, x) \pi_t(H, y),
\]
and again by Proposition~\ref{relationships}, we have
\[
\pi_s(G, x) \pi_t(H, y)
\leq \max_{x \in V(G)} \pi_s(G, x) \max_{y \in V(H)} \pi_t(H, y)
= \pi_s(G) \pi_t(H).
\]
\end{proof}

The proof of Proposition~\ref{all vertices implies graph} is similar
to that for Proposition~\ref{single distributions implies all sets};
however, Proposition~\ref{all vertices implies graph} is a
one-directional implication.  Since the sets of distributions used to
define $\pi(G)$ and $\pi_t(G)$ are not arbitrary, there is no easy way to
reverse the implication in Proposition~\ref{all vertices implies
graph} as there was in Proposition~\ref{single distributions implies
all sets}.

We now investigate equivalences within Conjecture~\ref{st-pebbling
vertices} involving different values of $s$ and $t$.  We show that if
Conjecture~\ref{st-pebbling vertices} holds for all graphs for a given
choice of $s$ and $t$, then it also holds if we double either $s$ or
$t$ and keep the other the same.  The basic idea of the proof is as
follows: given a graph $G$ and a target vertex $x_i$, we construct a
new graph $G_i'$ and choose a target vertex whose $s$-pebbling number
equals the $2s$-pebbling number of $x_i$ in $G$.  Then, given a target
vertex $y_j$ in a graph $H$, we compute the $2st$-pebbling number of
$(x_i, y_j)$ in $G \times H$ in terms of the $st$-pebbling number of
$(x', y_j)$ in $G_i' \times H$.  We begin by defining $G_i'$. \\
\textbf{Definition}: Given a graph $G$ and a vertex $x_i \in V(G)$, we
let $G_i'$ be the graph obtained by adding a single vertex $x'$ to
$V(G)$ and a single edge $(x_i, x')$.  Thus, $G$ is a subgraph of
$G_i'$. 

Now given another graph $H$, we define a function $\pi$ from
distributions on $G_i' \times H$ to distributions on $G \times H$. \\
\textbf{Definition}: Given a distribution $D$ on $G_i' \times H$, we
let $\pi(D)$ be the distribution on $G \times H$ obtained by replacing
every pebble on $(x', y)$ with two pebbles on $(x_i, y)$, i.~e.
\[ \pi(D)((x, y)) = \left\{ \begin{array}{rl}
D((x, y)) + 2 D((x', y)), & x = x_i \\ D((x, y)), & x \neq x_i
\end{array} \right. \]

\begin{proposition}
Let $D_0$ and $D_n$ be distributions on $G_i' \times H$ such that
$D_n$ is reachable from $D_0$.  Then a distribution that contains
$\pi(D_n)$ is reachable from $\pi(D_0)$ in $G \times H$.
\label{pulling back}
\end{proposition}

\begin{proof}
Suppose $D_0, D_1, \ldots, D_n$ is a sequence of distributions
obtained by pebbling moves from $D_0$ to $D_n$.  We show that for each
$m \geq 0$, a distribution that contains $\pi(D_m)$ is reachable from
$\pi(D_0)$.  Toward that end, suppose by induction that a distribution
containing $\pi(D_m)$ is reachable from $\pi(D_0)$, and consider the
move from $D_m$ to $D_{m+1}$.  That move replaces two pebbles from
some vertex $(x_1, y_1)$ with one pebble on an adjacent vertex $(x_2,
y_2)$.  If neither $x_1$ nor $x_2$ is $x'$, the same move in $G \times
H$ is a pebbling move from $\pi(D_m)$ to $\pi(D_{m+1})$.  If $x_1 =
x_2 = x'$, then $D_m((x', y_1)) \geq 2$, so $\pi(D_m)((x_i, y_1)) \geq
4$.  In this case, we use two pebbling moves to replace four pebbles
on $(x_i, y_1)$ with two pebbles on $(x_i, y_2)$, and these moves go
from $\pi(D_m)$ to $\pi(D_{m+1})$.

The only other cases to consider are pebbling moves from $(x_i, y)$ to
$(x', y)$, or from $(x', y)$ to $(x_i, y)$ for some vertex $y$.  In a
move from $(x_i, y)$ to $(x', y)$, we have $\pi(D_m) = \pi(D_{m+1})$,
and in a move from $(x', y)$ to $(x_i, y)$, we have $\pi(D_m) ((x_i,
y)) = \pi(D_{m+1}) ((x_i, y)) + 3$, since the two pebbles on $(x', y)$
contributed four to $\pi(D_m) ((x_i, y))$, and the single pebble on
$(x_i, y)$ only contributes one to $\pi(D_{m+1}) ((x_i, y))$.  We also
have $\pi(D_m) ((x, y)) = \pi(D_{m+1}) ((x, y))$ for all other
vertices, $x \neq x_i$.  Thus, $\pi(D_m)$ contains $\pi(D_{m+1})$.
\end{proof}

\begin{proposition}
For any graph $G$, any positive integer $s$, and any vertex $x_i \in
V(G)$, we have $\pi_{2s}(G, x_i) = \pi_s(G_i', x')$.
\label{2s-pebbling xi in G}
\end{proposition}

\begin{proof}
We consider a distribution $D$ on $G_i'$.  We assume
$x'$ is unoccupied in the original distribution, since replacing any
pebbles on $x'$ with two pebbles on $x_i$ does not help us reach $x'$
with additional pebbles.  Thus, we may consider $D$ to be a
distribution on the subgraph $G$ of $G_i'$, and we also have $\pi(D) =
D$ for every vertex $x \in V(G)$.

Now applying Proposition~\ref{pulling back} with $H$ equal to the
trivial graph, shows that if $D$ is a distribution on $G_i \times H
\cong G_i'$ from which $s$ pebbles can be moved to $x'$, then $\pi(D)
= D$ is also a distribution on $G \times H \cong G$ from which $2s$
pebbles can be moved to $x_i$.  The converse also holds; if we can put
$2s$ pebbles on $x_i$, we can then move $s$ pebbles to $x'$.  Thus, we
can move $2s$ pebbles onto $x_i$ if and only if we can move $s$
pebbles onto $x'$, and $\pi_{2s}(G, x_i) = \pi_s(G_i', x')$.
\end{proof}

\begin{theorem}
Suppose there are some values $s$ and $t$ such that $\pi_{st}(G \times
H, (x, y)) \leq \pi_s(G, x) \pi_t(H, y)$ for all graphs $G$ and $H$ and
all vertices $(x, y) \in V(G \times H)$.  Then $\pi_{2st} (G \times H,
(x, y)) \leq \pi_{2s}(G, x) \pi_t(H, y)$ and $\pi_{2st} (G \times H, (x, y))
\leq \pi_s(G, x) \pi_{2t}(H, y)$ for all graphs $G$ and $H$ and all
vertices $(x, y) \in V(G \times H)$.
\label{doubling}
\end{theorem}

\begin{proof}
Let $D$ be any distribution on $G \times H$ from which
$2st$ pebbles cannot be moved to the target vertex $(x_i, y_j)$.
Since $G \times H$ is a subgraph of $G_i' \times H$, we may also
regard $D$ as a distribution on $G_i' \times H$.  Furthermore, we have
$\pi(D) = D$.  By Proposition~\ref{pulling back}, we cannot move $st$
pebbles onto $(x', y_j)$ in $G_i' \times H$ from $D$.  Thus,
\[
\pi_{2st}(G \times H, (x_i, y_j)) \leq \pi_{st} (G_i' \times H, (x', y_j))
\leq \pi_s(G_i', x') \pi_t(H, y_j),
\]
and by Proposition~\ref{2s-pebbling xi in G}, $\pi_s(G_i', x') \pi_t(H, y_j)
= \pi_{2s} (G, x_i) \pi_t(H, y_j)$.
Similarly, we have
\[
\pi_{2st}(G \times H, (x_i, y_j)) 
\leq \pi_s(G, x_i) \pi_t(H_j', y') = \pi_s (G, x_i) \pi_{2t}(H, y_j),
\]
as desired.
\end{proof}

Motivated by this result, we make the following additional conjectures
as additional specializations of Conjecture~\ref{st-pebbling
vertices}.

\begin{conjecture}
For all graphs $G$ and $H$, all positive, odd integers $s$ and $t$,
and all vertices $x \in V(G)$ and $y \in V(H)$, we have $\pi_{st}(G
\times H, (x, y)) \leq \pi_s(G, x) \pi_t(H, y)$.
\label{odd-pebbling vertices}
\end{conjecture}

\begin{conjecture}
For all graphs $G$ and $H$, all nonnegative integers $a$ and $b$,
and all vertices $x \in V(G)$ and $y \in V(H)$, we have $\pi_{2^{a+b}}(G
\times H, (x, y)) \leq \pi_{2^a}(G, x) \pi_{2^b}(H, y)$.
\label{powers of two}
\end{conjecture}

\begin{theorem}
Conjecture~\ref{st-pebbling vertices} is equivalent to
Conjecture~\ref{odd-pebbling vertices}, and Conjecture~\ref{pebbling
vertices} is equivalent to Conjecture~\ref{powers of two}.
\label{two equivalences}
\end{theorem}

\begin{proof}
Conjecture~\ref{odd-pebbling vertices} clearly follows
from Conjecture~\ref{st-pebbling vertices}, and
Conjecture~\ref{pebbling vertices} follows from Conjecture~\ref{powers
of two} by letting $a = b = 0$.  In each case, the converse follows by
applying Theorem~\ref{doubling} inductively on the power of $2$ which
divides $st$.
\end{proof}

\section{Variants on Pebbling}
\label{variants}

Theorem~\ref{two equivalences} is interesting, but it would be more
satisfying to dispense with Conjectures~\ref{odd-pebbling vertices}
and~\ref{powers of two} and prove Conjectures~\ref{st-pebbling
vertices} and~\ref{pebbling vertices} are equivalent.  If we examine
the proof of Proposition~\ref{doubling} and Theorem~\ref{powers of
two}, we find that the powers of two in these results arise from the
rules of pebbling moves, and in particular, that two pebbles are
required from one vertex to put a pebble on an adjacent vertex.  In
Section~\ref{weighted pebbling}, we define pebbling on weighted
graphs.  We determine the cost of moving a pebble from one vertex to
an adjacent vertex by considering the weight of the edge between them.
Using these revised rules, we find that analogs of
Conjectures~\ref{st-pebbling vertices} and~\ref{pebbling vertices} are
indeed equivalent.

\subsection{Pebbling on Weighted Graphs}
\label{weighted pebbling}

In a \emph{weighted graph}, we attach positive integral weights to the
edges.  We use these weights to specify the cost of moving a pebble
from one vertex to another. \\
\textbf{Definition}: A \emph{weighted graph} is a graph $G = (V, E)$
together with a function $w : E(G) \rightarrow \mathbb{N}^+$.  We say
$w(e)$ is the \emph{weight} of the edge $e$. \\
\textbf{Definition}: A \emph{pebbling move along the edge $e = (x,
x')$ in a weighted graph} consists of removing $w(e)$ pebbles from
$x$, moving one of the pebbles onto $x'$, and throwing the other
pebbles away. 

We can then define each of the pebbling numbers $\pi(G, \mathcal{S})$,
$\pi(G, D)$, $\pi(G, v)$, $\pi(G)$, $\pi_t(G, v)$, $\pi_t(G)$, $\pi(G, t)$, and
$\gamma(G)$ for weighted graphs exactly as we did for unweighted
graphs.  We note that for this form of pebbling, any connected graph
may be regarded as a complete graph, since any missing edge $(v, w)$
may be added with a weight equal to the product of weights on some
path from $v$ to $w$.  We may also assume that the weight of each edge
is equal to the minimum of all such product; if there is an edge $e =
(v, w)$ for which this is not the case, we may use a path with a
smaller product to move a pebble from $v$ to $w$ instead of using $e$.

We show that the obvious analog of Sj\"{o}strund's Theorem
(Theorem~\ref{cover pebbling number}) is false by answering
Question~\ref{weighted cover pebbling} in the negative.

\begin{question}
If $G$ is a weighted graph, is $\gamma(G)$ the minimum number of
pebbles $N$ such that placing $N$ pebbles on a single vertex allows us
to cover $G$?
\label{weighted cover pebbling}
\end{question} 

\textbf{Answer}: No.  Consider the complete graph $K_4$ on vertices
$\{x_1, x_2, x_3, x_4 \}$ in which the weight of the edges $(x_1,
x_2)$ and $(x_3, x_4)$ is two, and the weight of every other edge is
five.  Then we can cover the graph if we start with thirteen pebbles
on any single vertex, however, we cannot cover the graph if we start
with nine pebbles on $x_1$ and four pebbles on $x_2$.

We now define the Cartesian product of two weighted graphs. \\
\textbf{Definition}: If $G$ and $H$ are two weighted graphs, their
Cartesian product is the weighted graph $G \times H$ whose vertex set
and edge set are the same as for the corresponding unweighted graph,
and whose weight function is given by
\[ \begin{array}{c}
w((x, y), (x, y')) = w(y, y') \mbox{ if } (y, y') \in E(H) \\
w((x, y), (x', y)) = w(x, x') \mbox{ if } (x, x') \in E(G).
\end{array} \]
We can now make each of the conjectures in Section~\ref{products} for
weighted graphs.  In each case, the conjecture on weighted graphs is
stronger than the corresponding conjecture on unweighted graphs, since
we can consider an unweighted graph to be a weighted graph in which
the weight of each edge is $2$.  We limit ourselves to the following
conjectures:

\begin{conjecture}
For all weighted graphs $G$ and $H$, all positive integers $s$ and
$t$, and all vertices $x \in V(G)$ and $y \in V(H)$, we have $\pi_{st}(G
\times H, (x, y)) \leq \pi_s(G, x) \pi_t(H, y)$.
\label{weighted st-pebbling vertices}
\end{conjecture}

\begin{conjecture}
For all weighted graphs $G$ and $H$ and all vertices $x \in V(G)$ and
$y \in V(H)$, we have $\pi(G~\times~H, (x, y)) \leq \pi(G, x) \pi(H, y)$.
\label{weighted pebbling vertices}
\end{conjecture}

Chung essentially proved Conjecture~\ref{weighted pebbling vertices}
when $G$ and $H$ are powers of $K_2$, i.~e.\ cubes in which the
weights of parallel edges are equal~(see \cite{Hypercubes},
Theorem~3).  We show Conjectures~\ref{weighted st-pebbling vertices}
and~\ref{weighted pebbling vertices} are equivalent; the proof is
similar to the proof of Theorem~\ref{doubling}.  We first modify the
required definitions. \\
\textbf{Definitions}: Given a weighted graph $G$, a positive integer
$s$, and a vertex $x_i \in V(G)$, we let $G_{i,s}'$ be the graph
obtained by adding a vertex $x'$ to $V(G)$ and a single edge $(x_i,
x')$ with weight $s$.  Given another graph $H$, we define the function
$\pi$ from distributions on $G_{i, s}' \times H$ to distributions on
$G \times H$ by replacing the pebbles on every vertex $(x', y)$ with
$s$ pebbles on $(x_i, y)$, i.~e.
\[ \pi(D)((x, y)) = \left\{ \begin{array}{rl}
D((x, y)) + s D((x', y)), & x = x_i \\ D((x, y)), & x \neq x_i
\end{array} \right. \]
for every distribution $D$ on $G_{i, s}' \times H$.

We give the analogs for Propositions~\ref{pulling back}
and~\ref{2s-pebbling xi in G} without proof.  The proofs are similar
to those of the original propositions.  We then prove
Theorem~\ref{weighted equivalence}.

\begin{proposition}
Let $D_0$ and $D_n$ be distributions on $G_{i,s}' \times H$ such that
$D_n$ is reachable from $D_0$.  Then $\pi(D_n)$ is reachable from
$\pi(D_0)$ in $G \times H$.
\label{weighted pulling back}
\end{proposition}

\begin{proposition}
For any weighted graph $G$, any positive integers $s$ and $t$, and any
vertex $x_i \in V(G)$, we have $\pi_{st}(G, x_i) = \pi_t(G_{i,s}', x')$.
\label{weighted st-pebbling xi in G}
\end{proposition}

\begin{theorem}
Conjectures~\ref{weighted st-pebbling vertices} and~\ref{weighted
pebbling vertices} are equivalent.
\label{weighted equivalence}
\end{theorem}

\begin{proof}
Conjectures~\ref{weighted st-pebbling vertices}
implies Conjecture~\ref{weighted pebbling vertices} by letting $s = t
= 1$.

Conversely, given a weighted graph $G$, a vertex $x_i \in V(G)$, and
an integer $s$, let $D$ be a distribution on $G \times H$ from which
$st$ pebbles cannot be placed on $(x_i, y_j)$.  Then by
Proposition~\ref{weighted pulling back}, we cannot place $t$ pebbles
on $(x', y_j)$ starting from $D$  in $G_{i, s}' \times H$, so
\[
\pi_{st}(G \times H, (x_i, y_j)) \leq \pi_t (G_{i,s}' \times H, (x', y_j)).
\]
Similarly, we form $H_{j, t}'$ by adding a vertex $y'$ and an edge
$(y_j, y')$ with a weight of $t$.  Then $\pi_t(H, y_j) = \pi(H_{j, t},
y')$ by Proposition~\ref{weighted st-pebbling xi in G}.  By
Proposition~\ref{weighted pulling back}, if $D'$ is a distribution on
$G_{i, s} \times H$ from which $t$ pebbles cannot be placed on $(x',
y_j)$, then in $G_{i, s}' \times H_{j, t}'$ we cannot place one pebble
on $(x', y')$ starting from $D'$.  Thus,
\[
\pi_t (G_{i,s}' \times H, (x', y_j)) \leq \pi(G_{i,s}' \times H_{j, t}',
(x', y')).
\]
But now if Conjecture~\ref{weighted pebbling vertices} holds for every
vertex in every graph, applying it gives 
\[
\pi_{st}(G \times H, (x_i, y_j)) \leq \pi(G_{i,s}' \times H_{j, t}', (x',
y')) \leq \pi(G_{i,s}', x') \pi(H_{j, t}', y'),
\]
and by Proposition~\ref{weighted st-pebbling xi in G}, we have
$\pi(G_{i,s}', x') \pi(H_{j, t}', y') = \pi_s(G, x_i) \pi_t(H, y_j)$,
as desired.
\end{proof}

\subsection{Target-selectable pebbling numbers}
\label{or-pebbling}

In this section we define a new pebbling number $\rho(G, \mathcal{S})$ and
investigate analogs of the conjectures in Section~\ref{products}.  As
with the definition of the usual pebbling number, we do not allow
ourselves to choose the starting distribution of $\rho(G, \mathcal{S})$
pebbles, but after those pebbles are placed, we allow ourselves to
choose which target distribution from $\mathcal{S}$ we wish to reach.
This definition was originally motivated by an attempt to prove a
version of Graham's conjecture.  We observe that if the vertex $v$ in
$G$ is unoccupied, we can move a pebble onto $v$ if and only if we can
move two pebbles onto some neighbor of $v$ in $G$, or equivalently, in
the graph obtained by deleting $v$ from $G$.  Therefore, it seems
reasonable to try to prove an analog of Graham's conjecture with this
pebbling number by using a form of induction on the number of vertices
in $G$.  However, we give simple counterexamples to show that $\rho(G,
\mathcal{S})$ does not satisfy what seems to be the natural analog to
Graham's conjecture. \\
\textbf{Definition}: Let $\mathcal{S}$ be a set of distributions on a
graph $G$.  Then the \emph{target-selectable pebbling number of
$\mathcal{S}$ in $G$}, denoted $\rho(G, \mathcal{S})$, is the smallest
number such that some distribution $D \in \mathcal{S}$ is reachable
from every distribution starting with $\rho(G, \mathcal{S})$ pebbles on
$G$.  We also define $\rho_t(G) = \rho(G, \mathcal{S}_t)$ and $\rho(G, v) =
\rho(G, \delta_v)$. 

We begin by formalizing our previous observation that $\pi(G, v) = \rho(G,
v)$ can be computed by determining how many pebbles are required to
put two pebbles on a neighbor of $v$.

\begin{proposition}
We have $\pi(G, v) = \rho(G, v) = \rho(G, N_2)$ where $N_2$ is the set of
distributions given by $N_2 = \{ 2 \delta_w : (v, w) \in E(G)
\}$.
\end{proposition}

We compute some values of $\rho(G, \mathcal{S})$ and relate them to the
usual pebbling number. 

\textbf{Observations}: Let $G$ be any graph with $n$ vertices.  Then:
\begin{enumerate}
\item We have $\rho_1(G) = 1$.  Thus, $\rho_1(G \times H) = \rho_1(G) \rho_1(H)$
for every graph $G$ and $H$, so the analog of Graham's conjecture for
the target-selectable pebbling number holds trivially.
\item We also have $\rho_2(G) = n+1$.  In particular, if $H$ has $m > 1$
vertices, then $\rho_2(G \times H) = mn+1 > \rho_2(G) \rho_1(H) = n+1$.  This
contradicts the analog of Conjectures~\ref{sets of distributions},
and~\ref{st-pebbling graphs} for the target-selectable pebbling
number.
\item For any distribution $D$ on $G$, we have $\rho(G, \{ D \}) = \pi(G,
\{ D \})$.  Thus, the analogs for Conjectures~\ref{single
distributions}, \ref{st-pebbling vertices}, and~\ref{pebbling
vertices} are equivalent to the original conjectures.
\end{enumerate}

We also note interesting relationships between this pebbling number
for paths and the regular pebbling number for cycles, as given by
Proposition~\ref{g(P)}.
We first define the distributions on the path that we are interested in. \\
\textbf{Definition}: 
Let the vertices on the path $P_n$ be $\{ x_1, x_2, \ldots, x_n \}$ in order.
Then we define $\mathcal{D}_t$ as the set of distributions
given by 
\[
\mathcal{D}_t = \{ t \delta_1, t \delta_n \}.  
\]

\begin{proposition}
If $n\ge 2$ and $i\ge 0$, we have $\rho(P_n,
\mathcal{D}_{2^i}) = \pi(C_{n+2i-1})$.
\label{g(P)}
\end{proposition}

\begin{proof}
For $i = 0$, we show $\rho(P_n,
\mathcal{D}_1) = \pi(C_{n-1})$.  Let the vertices of $C_{n-1}$ be $\{
y_1, y_2, \ldots, y_{n-1} \}$.  Given any distribution $D$ of pebbles
on $P_n$, let $D'$ be the distribution on $C_{n-1}$ given by
\[
D'(y_i) = \left\{ \begin{array}{ll}
D(x_1) + D(x_n) & \mbox{\ if\ } i = 1 \\
D(x_i) & \mbox{\ if\ } 2 \leq i \leq n-1
\end{array} \right.
\]
Then $y_1$ is reachable from $D'$ if and only if either $x_1$ or $x_n$
is reachable from $D$.  Thus, $\rho(P_n, \mathcal{D}_1) = \pi(C_{n-1})$.

To prove $\rho(P_n, \mathcal{D}_{2^i}) = \pi(C_{n+2i-1})$, let the vertices
of $C_{n+2i-1}$ be
\[
\{ z, a_{i-1}, \ldots, a_2, a_1, y_1, \ldots, y_n, b_1, b_2, \ldots,
b_{i-1} \}.
\]
Given a distribution $D$ of pebbles on $P_n$, let $D'$ be the
distribution on $C_{n+2i-1}$ given by $D'(y_i) = D(x_i)$ and $D'(v) =
0$ for every other vertex in $C_{n+2i-1}$.  Then $z$ is unreachable in
$C_{n+2i-1}$ if $2^i$ pebbles cannot be moved to either $x_1$ or $x_n$
in $P_n$.  Thus, $\rho(P_n, \mathcal{D}_{2^i}) \leq \pi(C_{n+2i-1})$.

To show that $\pi(C_{n+2i-1}) \leq \rho(P_n, \mathcal{D}_{2^i})$, we
construct a distribution of $\pi(C_{n+2i-1}) - 1$ pebbles in $P_n$ from
which $2^i$ pebbles cannot be moved to either $x_1$ or $x_n$.  We use
the critical distribution on $C_{n+2i-1}$.  Toward that end, if $n =
2k+1$, we have $\pi(C_{n+2i-1}) = 2^{k+i}$.  If we put $2^{k+i} - 1$
pebbles on $x_{k+1}$ we cannot move $2^i$ pebbles to either target.

On the other hand, if $n = 2k$, we analyze separately the cases when
$k+i$ is even or odd. 

\textbf{Case 1}: $k+i$ is even.  We let $m = \frac{k+i}{2}$.  Then
$n+2i = 2k + 2i = 4m$.  In this case, we have $\pi(C_{n+2i-1}) =
\pi(C_{4m-1}) = \frac{2^{2m+1} + 1}{3}$.  If we put $\frac{2^{2m} -
1}{3}$ pebbles each on $x_k$ and $x_{k+1}$, we have a total of
$\frac{2^{2m+1} - 2}{3} = \pi(C_{n+2i-1}) - 1$ pebbles on $P_n$.  Since
there are an odd number of pebbles on each vertex, one pebble cannot
be used in a pebbling move to the other vertex.  Therefore, at most
$\frac{2^{2m} - 4}{3}$ pebbles can be used, so we can transfer at most
$\frac{2^{2m-1} - 2}{3}$ pebbles from one occupied vertex to the
other.  Thus, we can put at most $\frac{2^{2m} - 1}{3} +
\frac{2^{2m-1} - 2}{3} = 2^{2m-1} - 1 = 2^{k+i} - 1$ pebbles on either
$x_k$ or $x_{k+1}$.  Thus, $2^i$ pebbles cannot be moved to either
$x_1$ or $x_n$. 

\textbf{Case 2}: $k+i$ is odd.  In this case, we let $m =
\frac{k+i-1}{2}$, so $n + 2i = 4m+2$.  Now $\pi(C_{n+2i-1}) =
\pi(C_{4m+1}) = \frac{2^{2m+2} - 1}{3}$.  If we put $\frac{2^{2m+1} -
2}{3}$ pebbles each on $x_k$ and $x_{k+1}$, we have $\pi(C_{4m+1}) - 1$
pebbles on $P_n$, and we can put at most $\frac{2^{2m+1} - 2}{3} +
\frac{2^{2m} - 1}{3} = 2^{2m} - 1 = 2^{k+i-1} - 1$ pebbles on either
$x_k$ or $x_{k+1}$.  Once again, $2^i$ pebbles cannot be moved to
either $x_1$ or $x_n$. 

Since we can always construct distributions of $\pi(C_{n+2i-1})$ pebbles
in $P_n$ from which no distribution of $\mathcal{D}_{2^i}$ is reachable,
we have $\pi(C_{n+2i-1}) \leq \rho(P_n, \mathcal{D}_{2^i})$.  Therefore,
$\pi(C_{n+2i-1}) = \rho(P_n, \mathcal{D}_{2^i})$, as desired.
\end{proof}

This gives rise to another counterexample to the analog of Graham's
conjecture.  If we let $T$ be the trivial graph with a single vertex
$v$ and let $\mathcal{S} = \{ 2 \delta_v \}$, it is natural to suppose
that a definition of multiplying distributions would give $\mathcal{S}
\cdot \mathcal{D}_1 = \mathcal{D}_2$.  Then, we would have
\[
\rho(T \times P_{4k+2}, \mathcal{S} \cdot \mathcal{D}_1) = \rho(P_{4k+2},
\mathcal{D}_2) = \pi(C_{4k+3}) = \frac{2^{2k+3} - (-1)^{2k+1}}{3} =
\frac{2^{2k+3} + 1}{3},
\]
and 
\[
\rho(T, \mathcal{S}) \rho(P_{4k+2}, \mathcal{D}_1) = 2 \pi(C_{4k+1}) = 2
\left(\frac{2^{2k+2} - (-1)^{2k}}{3} \right) = \frac{2^{2k+3} - 2}{3},
\]
but an analog of Graham's conjecture would require
\[
\rho(T \times P_{4k+2}, \mathcal{S} \cdot \mathcal{D}_1) \le
\rho(T, \mathcal{S}) \rho(P_{4k+2}, \mathcal{D}_1).
\]
contrary to what we would expect from an analog of Graham's
conjecture.

\bibliographystyle{plain}

\end{document}